\documentclass[11pt,a4paper]{article}
\usepackage{tikz}
\usepackage{subfigure}
\usepackage[english]{babel}
\usepackage{graphicx}
\usepackage{extarrows}
\usepackage{setspace}

\usepackage[ruled]{algorithm2e}                 %Ëã·¨ÅÅ°æÑùʽ1

\usepackage{indentfirst, latexsym, bm, amssymb}

\begin{document}

\newtheorem{theorem}{Theorem}[section]
\newtheorem{corollary}[theorem]{Corollary}
\newtheorem{definition}[theorem]{Definition}
\newtheorem{conjecture}[theorem]{Conjecture}
\newtheorem{question}[theorem]{Question}
\newtheorem{claim}{Claim}
\newtheorem{lemma}[theorem]{Lemma}
\newtheorem{prob}[theorem]{Problem}
\newtheorem{proposition}[theorem]{Proposition}
\newtheorem{example}[theorem]{Example}
\newenvironment{proof}{\noindent {\bf
Proof.}}{\hfill $\square$\par\medskip}
\newcommand{\remark}{\medskip\par\noindent {\bf Remark.~~}}
\newcommand{\pp}{{\it p.}}
\newcommand{\de}{\em}

\title{Stability in Bondy's theorem on paths and cycles}

\author{Bo Ning\footnote{College of Computer Science, Nankai University, Tianjin, 300350, P.R. China.
E-mail: bo.ning@nankai.edu.cn (B. Ning). Partially supported by NSFC (11971346).}~~ and~~  Long-Tu Yuan\footnote{
School of Mathematical Sciences,  Key Laboratory of MEA(Ministry of Education)  and  Shanghai Key Laboratory of PMMP,  East China Normal University, Shanghai 200241, China
Email: ltyuan@math.ecnu.edu.cn (L.-T Yuan). Partially supported by NSFC (12271169) and Science and Technology Commission
of Shanghai Municipality (22DZ2229014).}}
\date{}
\maketitle
\begin{abstract}
In this paper, we study the stability result of a well-known theorem of Bondy. We prove that for any 2-connected
non-hamiltonian graph, if every vertex except for at most one vertex has degree at least $k$, then it contains a cycle of length at least $2k+2$ except for some special families of graphs. 
Our results imply several previous classical theorems including a deep and old result by Voss. We point out our result on stability in Bondy's theorem can directly imply a positive solution (in a slight stronger form) to the following problem: Is there a polynomial time algorithm to decide whether a 2-connected graph $G$ on $n$ vertices
has a cycle of length at least $\min\{2\delta(G)+2,n\}$. 
This problem originally motivates the recent study on algorithmic aspects of Dirac's theorem by Fomin, Golovach, Sagunov and Simonov, although a stronger problem was solved by them by completely different methods. Our theorem can
also help us to determine all extremal graphs for wheels on odd number of vertices. We also discuss the relationship between our results and some previous problems and theorems in spectral graph theory and generalized Tur\'{a}n problem.
\end{abstract}


\vskip 0.5cm

\section{Introduction}
In this paper, we only consider graphs which are simple, undirected
and unweighed. Throughout this paper, $G$ denotes a graph.
A cycle in $G$ is called a \emph{Hamilton cycle} if it visits each vertex
of $G$ in cyclic order and only once. ``This is named after Sir
William Rowan Hamilton, who described, in a letter to his friend Graves
in 1856, a mathematical game on the dodecahedron in which one
person sticks pins in any five consecutive vertices and the other
is required to complete the path so formed to a spanning cycle
(see Biggs et al. (1986) or Hamilton (1931))." (see Bondy and
Murty \cite[pp.~471--472]{BM08}).

A fundamental theorem on Hamilton cycles in graph theory is
Dirac's theorem \cite{D52}: Every graph $G$ on $n$ vertices
has a Hamilton cycle if minimum degree $\delta(G)\geq n/2$.
One may want to improve the degree condition above, but the
family of graphs $K_{(n-1)/2,(n+1)/2}$ where $n$
is odd shows the theorem is sharp. However, if $\delta(G)$ is
not little compared with $n$, i.e., $\delta(G)=\Omega(n)$,
the cycle length function still seems to behave nicely.
For example, if $\delta(G)\geq n/k$ for some integer $k\geq 3$,
cycles of consecutive lengths or of given
lengths in $G$ were studied in \cite{TZ89,GHS02,HS15,KS20}. For a
comprehensive survey on classical aspects of Dirac's theorem
and its generalizations, we refer the reader to Li \cite{L13}.
Although nearly 70 years after Dirac's theorem appeared, this
area has been growing concern and still mysterious, for example,
see the recent algorithmic extensions of Dirac's theorem by
Fomin, Golovach, Sagunov and Simonov \cite{FGSS-Arxiv}.

But, in several situations $\delta(G)=o(n)$, and it is difficult
to find consecutive cycles or cycles of lengths for this case.
Instead, we can bound the length of the longest cycle in a
graph. To cite some results, we define the {\it circumference}
of $G$, denoted by $c(G)$, to be the length of a longest cycle
in $G$. For two given graphs $G$ and $H$, we use $G\cup H$ to
denote the vertex-disjoint union of copies of $G$ and $H$. Denote
$G+H$ by the graph obtained from $G\cup H$ and adding edges
between each vertex of $G$ and each vertex of $H$. Let $kH$ be
the graph consisting of $k$ vertex-disjoint copies of $H$.
The symbol
$\overline{G}$ denotes the complement of $G$.

In 1952, Dirac \cite{D52} proved the following fundamental theorem.

\begin{theorem}[Dirac \cite{D52}]\label{Thm:D52-longcycle}
Let $G$ be a 2-connected $n$-vertex graph. If $\delta(G)\geq k$ then $c(G)\geq \min\{2k,n\}$.
\end{theorem}

Ore \cite{O67} characterized all graphs with $c(G)=2k$ under the
condition of Theorem \ref{Thm:D52-longcycle}. Namely, Ore proved
that for any 2-connected non-Hamiltonian graph $G$ with
$\delta(G)=k$, $c(G)=2k$ if and only if
$\overline{K}_k+\overline{K}_s \subseteq G \subseteq K_{k}+\overline{K}_{s},s\geq k+1.$
Voss \cite{V91} improved Dirac's theorem by increasing the lower
bound of $c(G)$ by 2 and characterized all the classes of exceptional graphs.
We need the following several families of graphs to state Voss' theorem.

\begin{definition}
Define the graph $H(n,\ell,a)$ on $n$ vertices by taking a vertex
partition $A\cup B\cup C $ with $|A|=a, |B|=n-\ell+a$ and
$|C|=\ell-2a$ and joining all pairs in $(A,B)$ and all pairs
in $A\cup C$.
\end{definition}

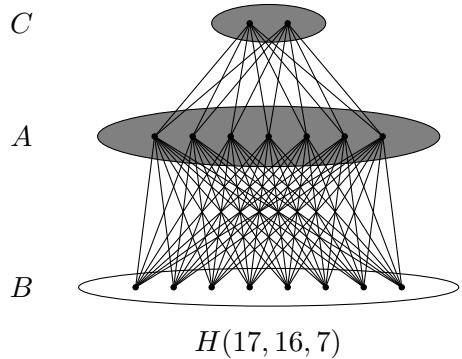
\begin{center}
\begin{tikzpicture}[scale = 0.5]

\filldraw[fill=gray] (3.5,4) ellipse (4.5 and 0.8);
\filldraw[fill=gray] (3.5,7) ellipse (1.5 and 0.5);
\draw  (3.5,0) ellipse (5 and 0.5);
 \foreach \xa / \xb  in {0/0,1/0,2/0,3/0,4/0,5/0,6/0,7/0,     0.5 / 4, 1.5 / 4, 2.5 / 4, 3.5 / 4, 4.5 / 4, 5.5 / 4, 6.5 / 4,    3/ 7, 4 / 7}
\filldraw[fill=black] (\xa , \xb) circle(2pt);
  \foreach \xa / \xb in {0 / 0, 1 / 0, 2 / 0, 3 / 0, 4 /0, 5/ 0, 6 / 0, 7 / 0}
    \foreach \ya / \yb in {0.5 / 4, 1.5 / 4, 2.5 / 4, 3.5 / 4, 4.5 / 4, 5.5 / 4, 6.5 / 4}
    {\draw (\xa,\xb) -- (\ya,\yb) ;}

   \foreach \xa / \xb in {3/ 7, 4 / 7}
    \foreach \ya / \yb in {0.5 / 4, 1.5 / 4, 2.5 / 4, 3.5 / 4, 4.5 / 4, 5.5 / 4, 6.5 / 4}
    {\draw (\xa,\xb) -- (\ya,\yb) ;}
\draw node at (-3,4){$A$};
\draw node at (-3,0){$B$};
\draw node at (-3,7){$C$};
\draw node at (3.5,-1.5){$H(17, 16,7)$};

\end{tikzpicture}
\end{center}

A direct observation shows that $H(n,\ell,a)$ does not contain a cycle of length at least $\ell$.

\begin{definition}
Let $F(t,k)=K_2+(tK_{k-1}\cup K_k)$ and $F_1(t,k)=K_2+(tK_{k-1}\cup K_k\cup K_1)$.
Define the graph $F(s,t,k)$ to be one obtained from three vertex-disjoint
graphs $K_3$, $sK_{k-1}\cup K_1$, and $tK_{k-1}$, in which $V(K_3)=\{x,y,z\}$,
$x$ and $z$ are adjacent to each vertex of $sK_{k-1}\cup K_1$, and $y$
and $z$ are adjacent to each vertex of $tK_{k-1}$.
\end{definition}

\medskip

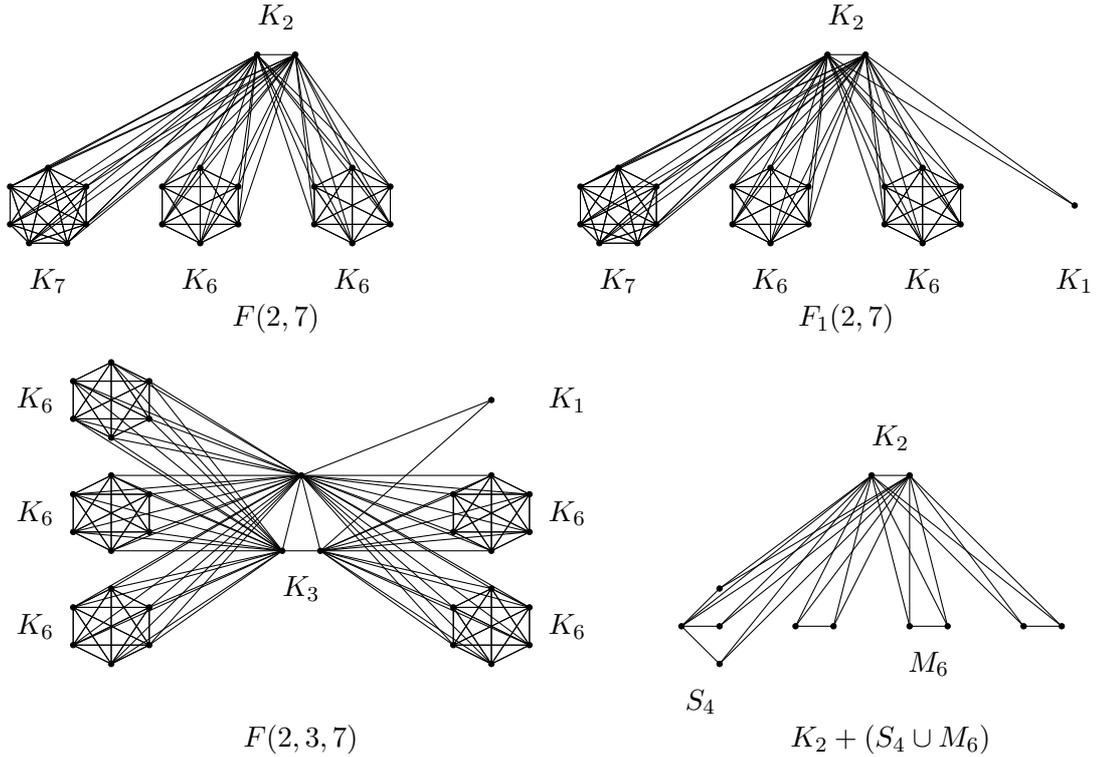
\begin{tikzpicture}[scale = 0.5]
\foreach \xa / \xb  in {0.5/0,-0.5/0,1/0.5,1/1.5,0/2,-1/0.5,-1/1.5,    4/0,5/0.5,5/1.5,4/2,3/0.5,3/1.5,     8/0,9/0.5,9/1.5,8/2,7/0.5,7/1.5,    5.5/5, 6.5/5  }
\filldraw[fill=black] (\xa , \xb) circle(2pt);

{\draw (5.5,5) -- (6.5,5) ;}

\foreach \xa / \xb in {0.5/0,-0.5/0,1/0.5,1/1.5,0/2,-1/0.5,-1/1.5}
\foreach \ya / \yb in {0.5/0,-0.5/0,1/0.5,1/1.5,0/2,-1/0.5,-1/1.5}
{\draw (\xa,\xb) -- (\ya,\yb) ;}

\foreach \xa / \xb in { 4/0,5/0.5,5/1.5,4/2,3/0.5,3/1.5}
\foreach \ya / \yb in { 4/0,5/0.5,5/1.5,4/2,3/0.5,3/1.5}
{\draw (\xa,\xb) -- (\ya,\yb) ;}

\foreach \xa / \xb in {8/0,9/0.5,9/1.5,8/2,7/0.5,7/1.5}
\foreach \ya / \yb in {8/0,9/0.5,9/1.5,8/2,7/0.5,7/1.5}
{\draw (\xa,\xb) -- (\ya,\yb) ;}

\foreach \xa / \xb in {5.5/5, 6.5/5 }
\foreach \ya / \yb in {0.5/0,-0.5/0,1/0.5,1/1.5,0/2,-1/0.5,-1/1.5,    4/0,5/0.5,5/1.5,4/2,3/0.5,3/1.5,     8/0,9/0.5,9/1.5,8/2,7/0.5,7/1.5 }
{\draw (\xa,\xb) -- (\ya,\yb) ;}

\draw node at (6,6){$K_2$};\draw node at (0,-1){$K_{7}$};\draw node at (4,-1){$K_{6}$};\draw node at (8,-1){$K_{6}$};
\draw node at (6, -2){$F(2,7)$};

\foreach \xa / \xb  in {14.5/0,15.5/0,16/0.5,16/1.5,15/2,14/0.5,14/1.5,    19/0,20/0.5,20/1.5,19/2,18/0.5,18/1.5,    23/0,24/0.5,24/1.5,23/2,22/0.5,22/1.5    ,   20.5/5, 21.5/5, 27/1 }
\filldraw[fill=black] (\xa , \xb) circle(2pt);

{\draw (20.5,5) -- (21.5,5) ;}

\foreach \xa / \xb in {14.5/0,15.5/0,16/0.5,16/1.5,15/2,14/0.5,14/1.5}
\foreach \ya / \yb in {14.5/0,15.5/0,16/0.5,16/1.5,15/2,14/0.5,14/1.5}
{\draw (\xa,\xb) -- (\ya,\yb) ;}

\foreach \xa / \xb in { 19/0,20/0.5,20/1.5,19/2,18/0.5,18/1.5}
\foreach \ya / \yb in {  19/0,20/0.5,20/1.5,19/2,18/0.5,18/1.5}
{\draw (\xa,\xb) -- (\ya,\yb) ;}

\foreach \xa / \xb in {23/0,24/0.5,24/1.5,23/2,22/0.5,22/1.5                       }
\foreach \ya / \yb in {23/0,24/0.5,24/1.5,23/2,22/0.5,22/1.5 }
{\draw (\xa,\xb) -- (\ya,\yb) ;}

\foreach \xa / \xb in {20.5/5, 21.5/5 }
\foreach \ya / \yb in {14.5/0,15.5/0,16/0.5,16/1.5,15/2,14/0.5,14/1.5,    19/0,20/0.5,20/1.5,19/2,18/0.5,18/1.5,    23/0,24/0.5,24/1.5,23/2,22/0.5,22/1.5, 27/1}
{\draw (\xa,\xb) -- (\ya,\yb) ;}

\draw node at (21,6){$K_2$};\draw node at (15,-1){$K_{7}$};\draw node at (19,-1){$K_{6}$};\draw node at (23,-1){$K_{6}$};\draw node at (27,-1){$K_1$};
\draw node at (21, -2){$F_1(2,7)$};
\end{tikzpicture}

\medskip

\begin{tikzpicture}[scale = 0.5]
%picture3
\foreach \xa / \xb  in {0/0,1/0.5,1/1.5,0/2,-1/0.5,-1/1.5,   0/3,1/3.5,1/4.5,0/5,-1/3.5,-1/4.5,     0/6,1/6.5,1/7.5,0/8,-1/6.5,-1/7.5,    10/0,11/0.5,11/1.5,10/2,9/0.5,9/1.5,   10/3,11/3.5,11/4.5,10/5,9/3.5,9/4.5,      10/7,       5/5,4.5/3,5.5/3           }
\filldraw[fill=black] (\xa , \xb) circle(2pt);

{\draw   (5,5)--(4.5,3);}{\draw   (5,5)--(5.5,3);}{\draw   (5.5,3)--(4.5,3);}

\foreach \xa / \xb in {0/0,1/0.5,1/1.5,0/2,-1/0.5,-1/1.5}
\foreach \ya / \yb in {0/0,1/0.5,1/1.5,0/2,-1/0.5,-1/1.5}
{\draw (\xa,\xb) -- (\ya,\yb) ;}

\foreach \xa / \xb in {0/3,1/3.5,1/4.5,0/5,-1/3.5,-1/4.5}
\foreach \ya / \yb in {0/3,1/3.5,1/4.5,0/5,-1/3.5,-1/4.5}
{\draw (\xa,\xb) -- (\ya,\yb) ;}

\foreach \xa / \xb in { 0/6,1/6.5,1/7.5,0/8,-1/6.5,-1/7.5}
\foreach \ya / \yb in { 0/6,1/6.5,1/7.5,0/8,-1/6.5,-1/7.5}
{\draw (\xa,\xb) -- (\ya,\yb) ;}

\foreach \xa / \xb in {10/0,11/0.5,11/1.5,10/2,9/0.5,9/1.5}
\foreach \ya / \yb in {10/0,11/0.5,11/1.5,10/2,9/0.5,9/1.5}
{\draw (\xa,\xb) -- (\ya,\yb) ;}

\foreach \xa / \xb in {10/3,11/3.5,11/4.5,10/5,9/3.5,9/4.5}
\foreach \ya / \yb in {10/3,11/3.5,11/4.5,10/5,9/3.5,9/4.5}
{\draw (\xa,\xb) -- (\ya,\yb) ;}

\foreach \xa / \xb in {5/5,4.5/3 }
\foreach \ya / \yb in {0/0,1/0.5,1/1.5,0/2,-1/0.5,-1/1.5,   0/3,1/3.5,1/4.5,0/5,-1/3.5,-1/4.5,     0/6,1/6.5,1/7.5,0/8,-1/6.5,-1/7.5 }
{\draw (\xa,\xb) -- (\ya,\yb) ;}

\foreach \xa / \xb in {5/5,5.5/3}
\foreach \ya / \yb in {10/0,11/0.5,11/1.5,10/2,9/0.5,9/1.5,   10/3,11/3.5,11/4.5,10/5,9/3.5,9/4.5,      10/7}
{\draw (\xa,\xb) -- (\ya,\yb) ;}

\draw node at (5,2){$K_3$};
\draw node at (-2,1){$K_{6}$};
\draw node at (-2,4){$K_{6}$};
\draw node at (-2,7){$K_{6}$};
\draw node at (12,1){$K_{6}$};
\draw node at (12,4){$K_{6}$};
\draw node at (12,7){$K_1$};

\draw node at (5, -2){$F(2,3,7)$};

%picture4

\foreach \xa / \xb  in { 20/5,21/5,      15/1, 16/1,16/2,16/0, 18/1,19/1, 21/1,22/1, 24/1,25/1         }
\filldraw[fill=black] (\xa , \xb) circle(2pt);
{\draw (20,5) -- (21,5) ;}{\draw (15,1) -- (16,1) ;}{\draw (15,1) -- (16,0) ;}{\draw (15,1) -- (16,2);}
{\draw (18,1) -- (19,1);}{\draw (21,1) -- (22,1);}{\draw (24,1) -- (25,1);}

\foreach \xa / \xb in {20/5,21/5}
\foreach \ya / \yb in {  15/1, 16/1,16/2,16/0, 18/1,19/1, 21/1,22/1, 24/1,25/1}
{\draw (\xa,\xb) -- (\ya,\yb) ;}
\draw node at (15.5,-1){$S_4$}; \draw node at (21.5,0){$M_{6}$}; \draw node at (20.5,6){$K_2$};

\draw node at (20.5, -2){$K_2+(S_4\cup M_{6})$  };
\end{tikzpicture}

\medskip

Let $S_s:=K_{1,s-1}$, i.e., the star on $s$ vertices, and $M_t$ be the graph on
$t$ vertices consisting of a matching with $\lfloor t/2 \rfloor$ edges
and one possible vertex (if $t$ is odd).

The following theorem can be found in a monograph
of Voss \cite{V91}.

\begin{theorem}[Voss \cite{V91}]\label{Thm:V77}
Let $G$ be a 2-connected non-Hamiltonian $n$-vertex graph with $c(G)\leq 2k+1$.
If $\delta(G)\geq k$ then $G$ is a subgraph of one of the following graphs:
\begin{itemize}
         \item $H(2k+1,2k+1,k)$, and $H(n,2k+2,k)$ with $n\geq 2k+2$;
         \item $F(s,t,k)$ with $s\geq 1$, $t\geq 1$, and $F(t,k)$ with $t\geq 1$; and
         \item $K_2+M_{t}$ with $t\geq 6$, $ K_2+(S_s\cup M_{t})$ with $s+t\geq 6$ when $k=3$,
         and $K_3+M_{t}$ with $t\geq 7$ when $k=4$.
       \end{itemize}
\end{theorem}

Among classical generalizations of Dirac's theorem (see \cite{L13}),
Bondy \cite{B72} strengthened Theorem~\ref{Thm:D52-longcycle} as follows.
In fact, Bondy try to solve a conjecture by Erd\H{o}s which determined
the Tur\'an number of a cycle of length $n-1$ and generalized it to
a general extremal problem, and the following theorem plays an important 
role in the whole process. 

\begin{theorem}[Bondy \cite{B72}]\label{Thm:B72}
Let $G$ be a 2-connected $n$-vertex graph on $n$. If every vertex
except for at most one vertex is of degree at least $k$ then
$c(G)\geq \min\{2k,n\}$.
\end{theorem}

The humble goal of this paper is to prove a stability result of
Bondy's theorem, which also improves Voss' theorem.

Our work is motivated by stability results from extremal graph
theory. In 1977, Kopylov \cite{K77} proved a very strong theorem:
if $G$ is an $n$-vertex 2-connected graph with $c(G)<k$,
then $e(G)\leq \max\{h(n,k,2),h(n,k,\lfloor (k-1)/2\rfloor)\}$,
where $h(n,k,a)=e(H(n,k,a))$.
Based on \cite{FKV16}, F\"{u}redi, Kostochka, Luo and Verstra\"{e}te \cite{FKLV18}
finally proved a stability result of Koplov's theorem in 2018.
For more work in this spirit, see \cite{CZ19,MN20,MY20+}.
Our main result can be seen as a solution to an analogous problem 
by considering the degree condition instead of the edge number condition
(see last section in \cite{FKV16}).

By a method quite different from Voss, we prove
the following theorem.
\begin{theorem}\label{Thm:stability-cycle}
Let $G$ be a 2-connected non-Hamiltonian $n$-vertex graph
with $c(G)\leq 2k+1$. If every vertex except for at most
one vertex is of degree at least $k\geq 2$, then $G$ is a
subgraph of one of the following graphs:
\begin{itemize}
\item $H(2k+1,2k+1,k)$,\footnote{Notice that $n=2k+1$ for this case.}
and $H(n,2k+2,k)$ with $n\geq 2k+2$;
\item $F(s,t,k)$ with $s\geq 1$, $t\geq 1$, and $F_1(t,k)$ with $t\geq 1$;
\item $K_2+M_{t}$ with $t\geq 6$,  and $ K_2+(S_s\cup M_{t})$
with $s+t\geq 6$ when $k=3$; and $K_3+M_{t}$ with $t\geq 7$ when $k=4$.
\end{itemize}
\end{theorem}

One of the most important motivations of the above theorem is that there
do not exist $n$-vertex $k$-regular graphs when $n$ and $k$ are both odd.
Hence, the graphs in this case which are close to $k$-regular graphs have
all vertices with degree $k$  except for exactly one vertex with degree $k-1$.
With the help of Theorem~\ref{Thm:stability-cycle}, we can characterize the
extremal graphs for wheels with odd number of vertices (odd wheels) in
Section \ref{Sec:Remark}.
The extremal graphs for odd wheels are important in the sense that we cannot delete an induced copy of linear forest in odd wheels to reduce its
chromatic number, that is, we cannot apply a deep theorem of Simonovits
to determine the extremal graph for odd wheels (see \cite{S74} and
\cite{Y22} for more information).

Our work is also motivated by interesting phenomenons
in spectral graph theory. Nikiforov and Yuan \cite{NY14}
tried to determine the graphs with maximum signless
Laplacian spectral radius among graphs of order $n$ without paths of given
length $k$. The main ingredient of their proof is a stability result
about graphs with large minimum degree and with no long paths.
Li and one of the authors \cite{LN16} studied the extremal graphs those
attain the maximum spectral radius among all graphs with minimum degree
at least $k$ but containing no Hamilton cycles. The main tool is
a stability result \cite[Lemma~2]{LN16} of a 1962 result of non-hamiltonian graph
due to Erd\H{o}s \cite{E62} (see also
F\"{u}redi, Kostochka and Luo \cite[Theorem~3]{FKL17}). We refer the reader
to the last section for corollaries
of our stability theorem which may have further applications to
subgraph problems in spectral graph theory.

Very recently, Zhu, Gy\H{o}ri, He, Lv, Salia and Xiao \cite{Z22} reproved Voss's Theorem and found many applications of Voss's Theorem in generalized Tur\'{a}n problems.
In this sense, our theorem would have some  applications in this area.

The last but not the least, Theorem \ref{Thm:stability-cycle}
can give a solution to the following algorithmic problem,
which is the original motivation of recent research of Fomin, Golovach, Sagunov and Simonov \cite{FGSS-Arxiv}. It should be mentioned that a
complete solution to a general version of the problem was
already given by Fomin et al. \cite{FGSS-Arxiv} by a
quite complicated and different method.
\begin{prob}[Fomin
et al. Question 1 in \cite{FGSS-Arxiv}]\label{Prob-1.4}
Is there a polynomial time algorithm to decide whether a 2-connected graph
$G$ on $n$ vertices contains a cycle of length at least $\min\{2\delta(G)+1,n\}$?
\end{prob}

It was commented in \cite{FGSS-Arxiv}
that ``The methods developed in the extremal Hamiltonian graph
theory do not answer this question.". We point out that
our proof of Theorem \ref{Thm:stability-cycle} is with aid of
the technique of ``vines of paths" (see p.~34--35 in  \cite{B03})
from extremal graph theory.

The paper is organized as follows. In Section \ref{Sec:Proof},
we present a complete proof of Theorem \ref{Thm:stability-cycle}.
In Section \ref{Sec:Remark}, we list some corollaries of our
result and discuss on an application of our result to algorithmic
aspects of long cycles.

\section{Proof of Theorem~\ref{Thm:stability-cycle}}\label{Sec:Proof}

Let $x$ be a vertex of a graph $G$. 
The {\it neighborhood} of $x$ in $G$ is denoted $N_G(x)=\{y\in V(G):xy\in E(G)\}$. 
If there is no danger of ambiguity,
we write it as $N(x)$ for simply. The {\it degree} of $x$ in $G$, denoted
by $d_{G}(x)$ (also $d(x)$ for simple), is the size of $N_G(x)$. For a path $P$ in $G$, denote
by $N_P(x)=N_G(x)\cap V(P)$, $d_P(x)=|N_P(x)|$ and
$N_P[x]=N_P(x)\cup \{x\}$. We use $x_iPx_j$ to denote the sub-path
$x_ix_{i+1}\ldots x_{j}$ of $P=x_1x_2\ldots x_k$ for $1\leq i<j\leq k$.
A Hamiltonian path of $G$ is a path which contains all vertices in $V(G)$.
For a subset $X\subset V(G)$, denote by $G[X]$  the subgraph
of $G$ induced by $X$.

We need the following lemma proved by Erd\H{o}s and Gallai \cite[Lemma~1.8]{EG59}.

\begin{lemma}[Erd\H{o}s and Gallai \cite{EG59}]\label{Lem:EG59}
Let $G$ be a graph and $x_1\in V(G)$ with $d(x_1)\geq 1$.
Suppose that the degree of every vertex of $G$
other than $x_1$ is at least $k\geq 2$.
Let $P=x_1x_2\ldots x_t$ be a path such that:
1) $P$ is a longest $x_1$-path; 2) subject to 1), $x_t$ has
a neighbor $x_s$ such that the distance between $x_s$
and $x_t$ along $P$ is largest among all $x_1$-paths.
Let $C=x_sPx_tx_s$. If $|V(C)|\leq 2k-1$, then
$V(C)=V(P)$ (i.e., $s=1$ and $x_1x_t\in E(G)$), or $G[V(C)]$ is an
end block (maximal 2-connected subgraph with at most one
cut-vertex) of $G$. Moreover, in both cases every two vertices
of $C$ is connected  by a Hamiltonian path of $G[V(C)]$ and
$|V(C)|\geq k+1$.
\end{lemma}

Bondy and Jackson \cite{BJ85} proved the following result, which was implicitly suggested
by Erd\H{o}s and Gallai (without proof, see Theorem 1.16 in \cite{EG59}).
\begin{theorem}[Bondy and Jackson \cite{BJ85}]\label{Thm:Bondy-Jackson}
Let $G$ be a 2-connected graph on at least $4$ vertices. Let $u$ and $v$
be two distinct vertices. If every vertex other than $u$, $v$
and at most another one vertex of $G$ is of degree at least $k$,
then $G$ has a $(u,v)$-path of length at least $k$.
\end{theorem}

We shall generalize Theorem \ref{Thm:Bondy-Jackson} by characterizing all extremal graphs with
somewhat different method. This result shall play an important role
in the proof of Theorem \ref{Thm:stability-cycle}.

\begin{lemma}\label{minimum degree without a-b path}
Let $k\geq 3$. Let $G$ be a 2-connected graph with two vertices $u$
and $v$ such that the longest $(u,v)$-path
is on at most $k+1$ vertices. If every vertex except for $u$, $v$ and
at most one more vertex of $G$, say $w$ (if existing), is of degree at least $k$, then
$G-\{u,v\}= \ell  K_{k-1}\cup K_1 $, or $G-\{u,v\}=\ell  K_{k-1}$,
for some $\ell\geq 1$. Moreover, $u$ and $v$ are adjacent to each
vertex of $V(G)-\{u,v,w\}$.
\end{lemma}
\begin{proof}
Let $G^{\prime}:=G-\{u,v\}$.
Suppose $k=3$. Let $H$ be a component of $G'$. If $H$
is 2-connected, then by Dirac's theorem, there is a cycle of length
at least 3 in $H$. By Menger's theorem, there are two vertex-disjoint paths
$P_1,P_2$ from $u,v$ to $w_1,w_2$ of $C$. Then $uP_1w_1Cw_2P_2v$
is a $(u,v)$-path of length at least 4, where $w_1Cw_2$ is a segment
of $C$ with length at least 2. If $H$ is separable with $|H|\geq 3$,
then we can find a $(u,v)$-path of length at least 4 by considering two
disjoint paths from $u,v$ to two end-blocks of $H$, respectively.
Thus, $H$ is a single vertex or an edge. Suppose that
$H$ is an edge. Since each vertex other than $u,v$ and $w$ has degree at least 3, $u$ and $v$
are adjacent to each vertex of $H$ if $H$ contains no $w$.
If $w\in V(H)$, then obviously $d(w)=2$
or $d(w)=3$, and the other vertex of $H$ has degree 3. If $|H|=1$,
then $V(H)=\{w\}$ and $u,v$ are adjacent to $w$ since $G$ is 2-connected.
In summary, $G'=\ell K_2$ or $G'=\ell K_2\cup K_1$ for some $\ell\geq 1$.
Moreover, $u$ and $v$ are adjacent to each
vertex of $V(G)-\{u,v,w\}$.
This proves the case of $k=3$.

In the following, assume $k\geq 4$. For this case, every vertex other than at most
one of $G^\prime$ is of degree at least $k-2\geq 2$. Let $H$ be a component
of $G'$. Let $x_1=w$ if $w\in V(H)$, and let $x_1$ be chosen arbitrarily
if $w\notin V(H)$. Choose a path $P=x_1x_2\ldots x_t$
such that: 1) $P$ is a longest $x_1$-path in $H$ starting
from $x_1$ and ending in $V(G)\setminus \{x_1\}$; 2) subject to 1),
$P$ is chosen such that $x_sx_t\in E(G)$ and $s$ is minimum among all $x_1$-paths.
Let $C=x_sPx_tx_s$, where $1\leq s \leq t-2$.

\medskip

(a). $|V(C)|\leq 2k-5$. If $d_{G^\prime}(x_1)\geq 1$, then by Lemma~\ref{Lem:EG59}, $V(P)=V(C)$
or $G[V(C)]$ is a terminal block with the cut-vertex $x_s$, and in both cases
every two vertices of $C$ are connected by a Hamiltonian path of $G[V(C)]$ and
$|V(C)|\geq k-1$.

We consider the following two cases:

\medskip

(a.1). $V(P)=V(C)$. Since $G$
is 2-connected, by Menger's theorem, $u$ and $v$ are adjacent to $C$
by two vertex-disjoint paths internally disjoint with $C$. Since every
two vertices of $C$ are connected
by a Hamiltonian path of $G[V(C)]$ and $|V(C)|\geq k-1$, $C$ is a
component of size at least $k-1$. Recall that the longest $(u,v)$-path
in $G$ contains at most $k+1$ vertices. Hence $H$ is a component of $k-1$
vertices; since otherwise there is some vertex of $H$ outside $C$,
and there is an $x_1$-path including all vertices of $C$ and at least one
vertex in $V(H)\backslash V(P)$, contradicting the choice of $x_1$-path.
Since there is at most one vertex of $G^\prime$ with degree
less than $k-2$ in $G^\prime$, $G^\prime[V(C)]$ is
a complete graph.

\medskip

(a.2). $G[V(C)]$ is an end-block with the
cut-vertex $x_s$ ($s\neq 1$). Since $G$ is 2-connected, $u$ and
$v$ are joint to $H-V(C)$ and $C-x_s$ by two independent edges,
respectively. Since every two vertices of $C$ are connected  by a
Hamiltonian path of $G[V(C)]$ and $|V(C)|\geq k-1$, one can easily
find a path starting from $u$ and ending at $v$ which
contains at least $k+2$ vertices, a
contradiction.

Therefore, for each case of (a.1) and (a.2), each component of $G^\prime$
is a complete graph on $k-1$ vertices or an isolated vertex (when
$d_{G^\prime}(x_1)=0$). Thus $G=K_2+(\ell  K_{k-1}\cup K_1)$ (if $d_{G^\prime}(w)=0$),
or $G-\{u,v\}=\ell  K_{k-1}$ (if $d_{G^\prime}(w)\neq 0$ or $w$ does not
exist) for some $\ell\geq 1$. For the latter case, $u$ and $v$ are adjacent to
each vertex of $V(G)-\{u,v,w\}$.

\medskip

(b). $|C|\geq 2k-3$. Since $G$ is 2-connected, by Menger's theorem,
$u$ and $v$ are adjacent to $C$ by two vertex-disjoint paths. Thus,
there is a path starting from $u$ and ending at $v$ on at least
$k+2$ vertices, a contradiction.

\medskip

(c). $|C|=2k-4$. Let $C=y_1y_2\ldots y_{2k-4}y_1$, where $y_1=x_s$
and $y_{2k-4}=x_t$. Let $y_r=y_{r'}$ when $r\equiv r'$ (mod) $2k-4$. Since $G$
is 2-connected, $u$ and $v$ are joint to $C$ by two vertex-disjoint
paths with two end-vertices $y_i$ and $y_j$. Without loss of generality,
assume $j>i$. Thus, we have $j-i=k-2$, since otherwise
there is a path from $u$ to $v$ on at least
$k+2$ vertices, a contradiction. By the choice of $P$ and $C$,
we have $N_{G^\prime}(y_2)\subseteq V(C)$ and
$N_{G^\prime}(y_{2k-4})\subseteq V(C)$. If $k\geq 5$, then
either $y_{i+1}y_{i+2}\ldots y_{j-1}$ or $y_{j+1}y_{j+2}\ldots y_{i-1}$
contains at least one vertex of $\{y_2,y_{2k-4}\}$. Without loss of
generality, let $y_2\in V(y_{i+1}y_{i+2}\ldots y_{j-1})$. Notice that
$y_{i+1}y_{i+2}\ldots y_{j-1}$ contains exactly $k-2$ vertices. Since
$d_{G}(y_2)\geq k$ and $y_2$ is nonadjacent to $u$ and $v$,
$y_2$ is adjacent to at least two vertices, say, $y^\prime$ and $y^{\prime\prime}$,
of $y_{j+1}y_{j+2}\ldots y_{i-1}$. Let $y^\prime$ precede $y^{\prime\prime}$
on $y_{j+1}y_{j+2}\ldots y_{i-1}$. Then either
$uy_iy_{i+1}\ldots y_2y^\prime\ldots y_{j-2}y_{j-1}y_jv$
or $uy_iy_{i-1}\ldots y^{\prime\prime}y_2\ldots y_{j+2}y_{j+1}y_jv$
is a path on at least $k+2$ vertices, a contradiction. Now let $k=4$.
Then $C=y_1y_2y_3y_4y_1$, where $y_1=x_s$. We may suppose that $u$ is
adjacent to $y_2$ and $v$ is adjacent to $y_4$, otherwise we are done by
a similar argument as before for $k\geq 5$. Then $y_1$ is a cut vertex of
$H$, since otherwise there is a $(y_1,y_2)$-path of $H$ containing
all vertices in $C$, and so there is a path from $u$ to
$v$ on at least 6 vertices, a contradiction. Now it follows that
$u$ or $v$ is adjacent to $H-V(C)$ since $G$ is
2-connected, and also creates a $(u,v)$-path on at least
6 vertices, a contradiction.

We finish the proof of the lemma.  \end{proof}

Now we are ready to prove Theorem~\ref{Thm:stability-cycle}.

\medskip

\noindent{\bf Proof of Theorem~\ref{Thm:stability-cycle}.}
The theorem holds trivially for $k=2$. Let $k\geq 3$.
By Theorem~\ref{Thm:B72}, if $n\leq 2k$, then $G$ is Hamiltonian, a contradiction.
Thus we may assume $n\geq 2k+1$.
Since the degree of each vertex of a graph does not decrease after
adding edges, we may further suppose that $G$ is a maximal non-Hamiltonian graph
with $c(G)\leq 2k+1$ (in the sense that the addition of any new
edge to $G$ creates a cycle of length at least $2k+2$ when $n\geq 2k+2$ and a cycle of length $2k+1$ when $n=2k+1$). 
The case for $n=2k+1$ is much easier than the case for $n\geq 2k+2$, hence we skip the proof for $n=2k+1$ and assume $n\geq 2k+2$.

Let $y$
be the unique vertex of $G$ with degree less than $k$ and if there
is no such vertex, we choose $y$ arbitrarily. Choose a maximum path $P=x_1x_2\ldots x_m$
such that the number of vertices of it with degree less than $k$
is minimum, that is, if there is a path on $m$ vertices without
containing $y$, then we will choose this path. 
Let $\mathcal{P}$ be the set of paths suitable for our choice.
By the maximality of $G$, we have $m\geq 2k+2$. Since $d_G(y)\geq 2$ (recall $G$ is 2-connected), we may choose $P$ with $d_G(x_1)\geq k$ and $d_G(x_m)\geq k$. 
In fact, if $y=x_1$, then $y$ is adjacent to $x_2$ and $x_i$.
Hence, we can consider the path $x_{i-1}Pyx_iPx_m$.
Let $N_P^-(x_1)=\{x_w:x_{w+1}\in N_P(x_1)\}$ and $N_P^+(x_{m})=\{x_w:x_{w-1}\in N_P(x_{m})\}$. 
Since $c(G)\leq 2k+1<m$, we have
\begin{equation}\label{neighbour of end vertices}
N_P^-(x_1)\cap N_P(x_{m})=\emptyset \mbox{ and }N_P^+(x_{m})\cap N_P(x_1)=\emptyset.
\end{equation}
Let $g=\max\{w:x_w\in N_P(x_1)\}$ and $h=\min\{w:w\in N_P(x_{m})\}$.
The proof of the following claim is standard.
\medskip

\noindent{\bf Claim.} $g\geq h$.

\medskip

\begin{proof} Since $G$ is 2-connected, there exists a path $Q_1$
such that it intersects $P$ with exactly two vertices $x_{s_1}$,
$x_{t_1}$ and $s_1<g<t_1$. Choose such path with $t_1$ as large
as possible. If $t_1>h$, then we stop. If $t_1<h$, then we choose
a path $Q_2$ such that it intersects $P$ with exactly two vertices
$x_{s_2}$, $x_{t_2}$ and $s_2<t_1<t_2$. Choose such a path with
$t_2$ as large as possible. Since we choose $t_1$ as large
as possible, we get $Q_1\cap Q_2 =\emptyset$ and $s_2\geq g$. If $t_2>h$,
then we stop. Otherwise, we may go on this procedure and
get a path $Q_{r}$ such that  it intersects $P$ with exactly
two vertices $x_{s_{r}}$, $x_{t_{r}}$ and $s_{r}<t_{r-1}\leq h<t_{r}$.
Moreover, for any $Q_i$ and $Q_j$ with $i<j$, either
$Q_i\cap Q_j =\emptyset$ or $Q_i\cap Q_j=s_i=t_{i+2}$ for $j=i+2$. Let
$$
i_0=\min\{w>s_1:x_w\in N_P(x_1)\}  \mbox{ and } j_0=\max\{w<t_r:x_w\in N_P(x_{m})\}.
$$
Let $r$ be odd. Since there is no $i^{\prime\prime}$ such that
$x_{i^{\prime\prime}}\in N_P(x_{1})$ and $x_{i^{\prime\prime}}\in N_P(x_{m})$,
$$
x_1Px_{s_1}Q_1x_{t_1}Px_{s_3}Q_3x_{t_3}\ldots x_{s_{r}}Q_{r}x_{t_{r}}Px_{m}x_{j_0}Px_{t_{r-1}}Q_{r-1}x_{s_{r-1}}\ldots x_{t_2}Q_2x_{s_2}Px_{i_0}x_1
$$
is a cycle of length at least $2k+2$, a contradiction. Let $r$ be even. Then
$$
x_1Px_{s_1}Q_1x_{t_1}P x_{s_3}Q_3x_{t_3} \ldots x_{s_{r-1}}Q_{r-1}x_{t_{r-1}}P x_{j_0}x_{m}P x_{s_{r}}Q_{r}x_{t_{r}}\ldots x_{t_2}Q_2x_{s_2}P x_{i_0}x_1
$$
is a cycle of length at least $2k+2$, a contradiction. The proof
of the claim is complete.\end{proof}

Let $U=\{x_1\}\cup N_P(x_1)\cup N_P^+(x_{m})\setminus \{x_{i+1}\}$.
By $d_G(x_1)\geq k$, $d_G(x_m)\geq k$ and (\ref{neighbour of end vertices}), $|U|\geq 2k$.
Let $V_1=V(x_1Px_i)$ and $V_2=V(x_jPx_m)$.
Clearly, $x_1Px_ix_mPx_jx_1$ is a cycle of length at most $2k+1$, implying $|V_1\cup V_2|\leq 2k+1$.
Since $U \subseteq V_1\cup V_2$, we have  $|(V_1\cup V_2)\setminus U|\leq 1$.

We consider the following two cases.

\medskip

\noindent{\bf Case 1.} There exists a pair $(i,j)$ ($i<j$) satisfying the following:
\begin{equation}\label{i and j}
x_i\in N_P(x_{m}), x_j\in N_P(x_1) \mbox{ and } x_{w}\notin N_P(x_1)\cup N_P(x_m) \mbox{ for each }1\leq i<w<j\leq m.
\end{equation}
Let
$$s=\min\{w:x_{w}\in N_P(x_m)\} \mbox{ and } t=\max\{w:x_{w}\in N_P(x_1)\}.$$

We consider the following two subcases:

\medskip

\noindent{\bf Subcase 1.1.} For each maximum path in $\mathcal{P}$, there is at most one pair $(i,j)$ satisfying (\ref{i and j}).

\medskip

Subject to our choice of $P$, we choose $P$ with only one pair $(i,j)$ satisfying (\ref{i and j}) such that $j-i$ as small as possible.
Let $k=3$ and $j-i=2$. Then $P=x_1x_2\ldots x_8$.
Without loss of generality, let $i=4$ and $j=6$. Since $c(G)\leq 7$,
$x_5$ is not adjacent to $x_1$, $x_3$, $x_7$ and $x_8$. We claim
that $G-V(P)$ is an independent set and each vertex of $G-V(P)$ can only be adjacent to $x_2$, $x_4$ and $x_6$.
Otherwise, we can easily find a longer path, a contradiction.
Let $n\geq 9$.
Since there are at most one vertex with degree two in $G$, either $z\in V(G)-V(P)$ or $x_5$ is adjacent to all of  $x_2$, $x_4$, $x_6$.
Hence, $x_1,x_3,x_5$, the edge $x_7x_8$ can only be adjacent to $x_2,x_4,x_6$ (otherwise there is a cycle of length 8) and $G=H(n,8,3)$ by the maximality of $G$.
Thus  $x_1$ is adjacent to $x_4$ and $x_8$ is adjacent to $x_2$, contradicting the fact $P$ has only one pair $(i,j)$ satisfying (\ref{i and j}). 
Let $n=8$. 
If $x_2$ is adjacent to $x_5$, then by $c(G)\leq 7$, there is no edge between $x_1,x_3,x_5$ and $x_7,x_8$.
Hence $G=H(8,8,3)$, a contradiction as before.
Thus $d_G(x_5)=2$ and other vertices in $G$ have degree at least $3$.
If $x_8$ is adjacent to $x_2$, then since $P$ has only one pair $(i,j)$ satisfying (\ref{i and j}), $x_1$ is not adjacent to $x_4$.
Moreover, by $c(G)\leq 7$, $x_1$ is not adjacent to $x_3,x_5,x_7,x_8$, implying $d_G(x_1)\leq 2$, a contradiction.
Hence $x_8$ is not adjacent to $x_2$.
Similarly, $x_8$ is not adjacent to $x_3$ (otherwise $d_G(x_7)\leq 2$, a contradiction), whence $x_8$ is only  adjacent to $x_4,x_6$ and $x_7$.
Consider the path $x_3x_2x_1x_6x_5x_4x_8x_7$. Since $x_7$ is adjacent to $x_6$ and $x_3$ is adjacent to $x_4$, $x_7$ is only adjacent to $x_4,x_6$ and $x_8$.
Thus $G=K_2+(K_3\cup K_2\cup K_1)$ by the maximality of $G$ ($x_1,x_2,x_3$ induce a triangle in $K_3\cup K_2\cup K_1$).

Now, let $k\geq 4$ or $j-i\geq 3$.
We will prove the following two claims.

\medskip

\noindent {\bf Claim 1.} $s=i$ and $j=t$.

\medskip

\begin{proof} Let $j-i=2$ and $k\geq 4$. Then $m=2k+2$, otherwise $x_1Px_ix_mPx_jx_1$ is a cycle on at least $2k+2$ vertices, a contradiction.
Moreover, $x_{i+1}$ is nonadjacent to any vertex of $G-V(P)$; otherwise there is a path on at least $m+1$ vertices, a contradiction.
Let $x_{\ell}\in N_P^-(x_1)\cap V_1$. Then $x_{i+1}$
is nonadjacent to $x_{\ell}$. Otherwise,
$x_{i+1}Px_mx_iP x_{\ell+1}x_1Px_{\ell}x_{i+1}$ is a cycle on $2k+2$
vertices, a contradiction. Similarly, $x_{i+1}$ is nonadjacent to any
vertex of $(N_P^-(x_1)\cup N_P^-(x_m))\cap V_1$ and
$(N_P^+(x_1)\cup N_P^+(x_m))\cap V_2$. Since $d_G(x_1)=d_P(x_1)\geq k$,
$d_G(x_m)=d_P(x_m)\geq k$ and $m=2k+2$, recall that there is only one pair $(i,j)$ satisfying (\ref{i and j}), we have
$$|(V_1\setminus (N_P^-(x_1)\cup N_P^-(x_m)))   \cup (V_2\setminus(N_P^+(x_1)\cup N_P^+(x_m)))|\leq 3.$$
Thus, we have $d_G(x_{i+1})\leq 3<k$ and $x_{i+1}=y$. Hence, each vertex of
$V(P)\setminus \{x_{i+1}\}$ is of degree at least $k$ in $G$.

Since $U\subseteq V_1\cup V_2$,
$|V_1\cup V_2| \leq 2k+1$ and $|U|\geq 2k$, without loss of
generality, we may suppose that $V_2\subseteq U$, i.e., $V_2\cap U=V_2$. Suppose that $x_m$
is adjacent to consecutive vertices $x_{\ell}$ and $x_{\ell+1}$ of $V_1$.
Let $P^\prime=x_{t+1}Px_mx_tPx_1$.
Note that the neighbours of $x_1$ in $P^\prime$ are determined.
By (\ref{neighbour of end vertices})
and $d_G(x_{t+1})=d_{P^\prime}(x_{t+1})\geq k$, without loss of generality, $x_{t+1}$ is
adjacent to  $x_{\ell}$. Thus, $x_mx_{\ell+1}Px_jx_1Px_{\ell}x_{t+1}Px_m$
is a cycle on $2k+2$ vertices, a contradiction. Hence, $x_m$ 
cannot be adjacent to both of consecutive vertices of $V_1$ (keep this proof in mind, we
will frequently use the idea of this proof).

We will first show that $s=i$.
From the last paragraph, $x_m$ is not adjacent to $x_{i-1}$.
Suppose that $i\neq s$.
Since there is only one pair $(i,j)$ satisfying (\ref{i and j}), $x_1$ is not adjacent to $x_i$.
By $|V_1\setminus U|\leq 1$, $x_m$ is adjacent to $x_{i-2}$ and $x_1$ is adjacent to each vertex of $V(x_2Px_{i-2})$.
Consider the path $x_{\ell}Px_mx_{\ell-1}Px_1$ for $j+1\leq\ell\leq m-1$.
By (\ref{neighbour of end vertices}) and the fact $y\notin V(x_jPx_m)$, as in the previous proof, $x_\ell$ is adjacent to two non-consecutive vertices of $\{x_{i-2},x_{i-1},x_i\}$, that is $x_\ell$ is adjacent to   $x_{i-2}$ and $x_i$.
Thus for $j+1\leq\ell\leq m$, $x_{\ell}$ is adjacent to each vertex of $V(x_iPx_m)\setminus \{x_\ell\}$.
Moreover,  $x_{i-1}$ is nonadjacent to
$V(x_1Px_{i-3})\cup\{x_{i+1}\}$, otherwise,
it is not hard to find a cycle on $2k+2$ vertices, a contradiction.
Furthermore, each vertex of $G-V(P)$ is nonadjacent to $x_{i-1}$,
otherwise there is a path on at least $m+1$ vertices, a contradiction.
Thus $d_G(x_{i-1})\leq k-1$, contradicting that there is at most
one vertex of $G$ with degree less than $k$. We conclude that $s=i$.

Now, we will show that $j=t$. We will show that $x_1$ cannot be adjacent
to both of consecutive vertices of $V_2$ in the following two cases:

(1.a). $x_i\in U$, i.e., $x_i\in N_P(x_1)$. 
Consider the path $x_{i-1}Px_1x_iPx_m$.
Since $d_G(x_{i-1})\geq k$, by an argument similar as the previous one (consider the neighbors of $x_{i-1}$ and $x_1$ in $P$), $x_1$ cannot be adjacent
to both of  consecutive vertices of $V_2$.

(1.b). $x_i\notin U$, i.e., $x_i\notin N_P(x_1)$.
Then $|V_1\cup V_2|=2k+1$ and $d_G(x_1)=d_G(x_m)=k$.
Suppose that $x_1$ is adjacent to consecutive vertices of $V_2$.
As $V_2\cap U=V_2$, $x_1$ is adjacent to all vertices of $x_{j}Px_{t}$.
First, $x_{i-1}$ is nonadjacent to  $z\in G-V(P)$, otherwise $zx_{i-1}Px_1x_{j}Px_ix_mPx_{j+1}$ is a path
on $m+1$ vertices, a contradiction.
Also, $x_{i-1}$ is nonadjacent to $x_{i+1}$;
otherwise, there is a cycle $x_{i-1}Px_1x_jPx_mx_ix_{i+1}x_{i-1}$ on $2k+2$ vertices, a contradiction.
Next, we will show that $N_P[x_w]=N_P[x_m]$ for $t+1\leq w\leq m-1$.
Suppose not and choose $q$ maximum such that $N_P[x_q]\neq N_P[x_m]$.
Thus for $x_{q+1}\in V(x_{t+1}Px_m)$, we have $N_P[x_{q+1}]=N_P[x_m]$.
Considering the path $x_{w}Px_mx_{w-1}Px_1$  for $t+1\leq w\leq m-1$, by (\ref{neighbour of end vertices}), $x_w$ can only be adjacent to $x_{i-1}$ in $V(x_1Px_{i-1})$, hence $x_q$ is adjacent to $x_{i-1}$.
We can find the cycle $x_qPx_{t+1}x_{q+1}Px_mx_iPx_tx_1Px_{i-1}x_q$ is a cycle on $2k+2$ vertices, a contradiction.
Therefore, we have $N_P[x_w]=N_P[x_m]$ for $t+1\leq w\leq m-1$, implying that $x_{i-1}$ is nonadjacent to each vertex of $V(x_{t+1}Px_m)$.
The paths $x_{i-1}Px_1x_wPx_ix_mPx_{w+1}$ and $x_{i-1}Px_1x_wPx_mx_jPx_{w-1}$ for $j \leq w\leq t$ imply that $x_{i-1}$ is not adjacent to any vertices of $x_jPx_t$.
Hence, $d_G(x_{i-1})\leq k-1$ ($d_G(x_1)=k$, $x_i\notin N_P(x_1)$ and $|V(x_jPx_t)|\geq 2$), a contradiction. 
Thus $x_1$ cannot be adjacent to both of consecutive vertices of $V_2$.

Since $V_2\subseteq U$, in both of (a) and (b), we have $j=t$ (recall that there is only one pair $(i,j)$ satisfying (\ref{i and j})). We
finish the proof of the claim for the case when $j-i=2$ and $k\geq 4$.

\medskip

Let $j-i\geq 3$.
We consider the following two cases:

\medskip

(2.a). $U=V_1\cup V_2$. Suppose that $x_m$ is adjacent to consecutive
vertices $x_{\ell}$ and $x_{\ell+1}$ of $V_1$. Considering the
path $x_{t+1}Px_mx_tPx_1$ (or $x_{t+2}Px_mx_{t+1}Px_1$ \footnote{It is possible that $x_{t+2}=x_m$.} when
$y=x_{t+1}$), by (\ref{neighbour of end vertices}), $x_{t+1}$ (or $x_{t+2}$) is adjacent to at least
one of $x_{\ell}$ and $x_{\ell+1}$. Thus, $G$ contains a cycle
$x_{t+1}x_{\ell+1}Px_tx_1Px_{\ell}x_mPx_{t+1}$
(or $x_{t+2}x_{\ell+1}Px_tx_1Px_{\ell}x_mPx_{t+2}$) of
length at least $2k+2$ (note that $j-i\geq 3$), a contradiction.
Hence $x_m$ cannot be adjacent to both of consecutive vertices of $V_1$.
Since $V_1\subseteq U$ and there is only one pair satisfying (\ref{i and j}), we have $s=i$. Similarly, we have $j=t$.

\medskip

(2.b). $|U|=|V_1|+|V_2|-1$, i.e., $|V_1\cup V_2|=2k+1$. Without loss of generality, let
$\{x_h\}=V_1\setminus U$. We will first show that $s=i$. As case (a), $x_m$ is nonadjacent
to consecutive vertices of $V_1$. Suppose that $x_m$ is
adjacent to $x_i$ and $x_{i-2}$  (otherwise, we have $s=i$) and $x_1$ is adjacent to
each vertex of $V(x_2Px_{i-2})$.
Considering the path $x_{i-3}Px_1x_{i-2}Px_m$ (or
$x_{i-4}Px_1x_{i-3}Px_m$ when $x_{i-3}=y$),  as the
previous argument in (2.a), $x_1$ cannot be adjacent to both of consecutive
vertices of $V_2$. Thus $j=t$ (here we suppose that
$s\neq i$). Hence, $x_{j+1}$ is nonadjacent to $x_{i-1}$,
otherwise $x_{i-1}Px_jx_1Px_{i-2}x_mPx_{j+1}x_{i-1}$ is
a cycle on at least $2k+2$ vertices, a contradiction.
Now considering the path $x_{j+1}Px_mx_jPx_1$, since $|V_1\cup V_2|=2k+1$,
$x_{j+1}$ is adjacent to two non-consecutive vertices of $\{x_i,x_{i-1},x_{i-2}\}$, that is both of $x_i$ and $x_{i-2}$.
Hence $x_{i-2}Px_1x_jPx_ix_mPx_{j+1}x_{i-2}$ is a cycle
on at least $2k+2$ vertices (note that $j-i\geq 3$).
This contradiction shows that $s=i$.

Form the last paragraph, $N_P[x_1]\cup \{x_h\}=V(x_1Px_i)\cup V(x_jPx_t)$ and $N_P[x_m]=V(x_tPx_m)\cup \{x_i\}$.
Suppose that $j<t$.
Then $x_1$ is adjacent to all vertices of $x_jPx_{t}$.
Consider the path $x_\ell Px_mx_{\ell-1}Px_1$ for $t+1\leq \ell\leq m-1$.
The vertex $x_\ell$ cannot be adjacent to consecutive vertices of $V_1$.
Recall that $|V_1\cup V_2|=2k+1$.
If $h=i$, then since we choose $j-i$ as small as possible, $x_\ell$ can only be adjacent to $x_i$ in $V(x_1Px_j)$ (let $x_q$  be adjacent to $x_{i-1}$ with $t+1\leq q\leq m-1$ and $q$ is maximum, we can find a cycle of length at least $2k+2$ as in $(1.b)$);
if $h\neq i$, then $x_\ell$ also can only be adjacent to $x_i$ in $V(x_1Px_j)$  (if $x_\ell$ is adjacent to $x_{h-1}$, then $x_{h-1}Px_1x_{h-2}Px_{\ell-1}x_mPx_\ell x_{h-1}$ is a cycle of length at least $2k+2$, a contradiction).
Thus $N_P[x_m]=N_P[x_\ell]$ for $t+1\leq \ell\leq m-1$, implying that $G[V(x_{t+1}Px_m)\cup\{x_i\}]$ is a complete graph.

If $h\neq i$, considering the path
$x_{i-1}Px_1x_iPx_m$ (or $x_{i-2}Px_1x_iPx_m$\footnote{$x_{i-2}Px_1x_iPx_m$ is a maximal path, i.e., all neighbors of $x_{i-2}$ and $x_m$ belongs to this path, otherwise we find a path of length $m$ without containing $y$, a contradiction.} when
$x_{i-1}=y$), then $x_{i-1}$ (or $x_{i-2}$) is adjacent to one of $x_j$ and $x_{j+1}$.
Without loss of generality, let $x_{i-1}$ (or $x_{i-2}$) be adjacent to $x_{j+1}$.
Hence,  $x_{i-1}x_{j+1}Px_mx_iPx_jx_1Px_{i-1}$ or $x_{i-2}x_{j+1}Px_mx_iPx_jx_1P x_{i-2}$ is a cycle on at least $2k+2$ vertices ($j-i\geq 3$), a contradiction.

If $h=i$, considering the path $x_{i-1}Px_1x_tPx_ix_mPx_{t+1}$
(or $x_{i-2}Px_1x_tPx_ix_mPx_{t+1}$  when $x_{i-1}=y$),
then, without loss of generality, $x_{i-1}$ is adjacent to $x_{j+1}$.
As in the case $h\neq i$, we can also get a cycle on at least $2k+2$ vertices, a contradiction.
The proof of the claim is complete.\end{proof}

\medskip

\noindent {\bf Claim 2.}
Each vertex in $G-V_1\cup V_2$ is
nonadjacent to $(V_1\cup V_2)\setminus \{x_i,x_j\}$.

\medskip

\begin{proof}
Let $z\in G-V_1\cup V_2$.
If $U=V_1\cup V_2$, then the assertion holds trivially;
otherwise there is a path on $m+1$ vertices when $z\in G-V(P)$ or a
cycle on $m$ vertices when $z\in V(x_{i+1}Px_{j-1})$, a contradiction.
Without loss of generality, let $x_h\in V_1\setminus U$, i.e., $x_h\notin N_P(x_1)$.
As in the above argument, $z$ is nonadjacent to $(V_1\cup V_2)\setminus \{x_i,x_j,x_{h-1}\}$.
Suppose that $z$ is adjacent to $x_{h-1}$.
We divide the proof into the following two cases.

\medskip

(3.a). $z\in G-V(P)$.
(3.a.1). $h=i$. Then $zx_{i-1}Px_1x_jPx_ix_mPx_{j+1}$ is a path on $m+1$ vertices,
a contradiction.
(3.a.2) $h\neq i$. Then $d_P(x_h)\geq k$, since otherwise
$zx_{h-1}Px_1x_{h+1}Px_m$ is a path on $m$ vertices without
containing $y$, contradicting the choice of $P$.
Thus, considering the path $x_hPx_1 x_{h+1}Px_m$, since the neighbors of $x_m$ in this path are determined, by  (\ref{neighbour of end vertices}) we have $N_P(x_1)=N_P(x_h)$, and so   $zx_{h-1}x_hx_{h-2}Px_1x_{h+1}Px_m$
is a path on $m+1$ vertices, a contradiction.

\medskip

(3.b). $z\in V(x_{i+1}Px_{j-1})$. If $h=i$, then
$zPx_ix_mPx_jx_1x_{h-1}z$ is a cycle on $2k+2$ vertices,
a contradiction. Now we may suppose that  $h\neq i$.
(3.b.1). $k\geq 4$ and $j-i=2$.
Then $d_G(x_{i+1})\leq k-1$ as in the proof of Claim 1,
and so $x_{i+1}=y$. So $d_P(x_h)\geq k$ and  $x_h$ is adjacent
to $x_{h-2}$. Thus $x_{i+1}x_{h-1}x_hx_{h-2}Px_1x_{h+1}Px_ix_mPx_{i+1}$
is a cycle on at least $2k+2$ vertices, a contradiction.
(3.b.2) $j-i\geq 3$.
Then either $zx_{h-1}Px_1x_hPx_ix_mPz$ or  $zPx_ix_mPx_jx_hPx_{i-1}x_1Px_{h-1}z$ (as in (3.a), we have $N_P(x_1)=N_P(x_h)$)
is a cycle on at least $2k+2$
vertices, a contradiction.
The proof is complete.
\end{proof}

By the maximality of $P$, the longest
path starting from $x_i$ through $G-V_1\cup V_2$ ending
at $x_j$ is on at most $j-i+1$ vertices.
If $i=k+1$, then since $c(G)\leq 2k+1$, we have $j-i+1\leq k+1$.
Let $i=k.$
Then we still have
$j-i+1\leq k+1$. Otherwise, if $j-i+1\geq k+2$, then
$x_{i+1}Px_mx_iPx_1$ (or $x_{j-1}Px_1x_jPx_m$ when $x_{i+1}=y$)
with the pair $(x_i,x_j)$ contradicts the choice of $(i,j)$
(recall we choose $j-i$ as small as possible).
Thus, by Claim 2, it is easy to see that
$G^\ast=G-V_1\cup V_2 \setminus \{x_i,x_j\}$ is 2-connected
(note that $x_ix_j\in E(G)$ by the maximality of $G$) such that the longest
path starting from $x_i$ ending at $x_j$ is on at most
$k+1$ vertices. Now applying
Lemma~\ref{minimum degree without a-b path}
for $G^\ast$, we have $G=K_2+(t  K_{k-1}\cup K_1)$
or $G= K_2+((t-1) K_{k-1}\cup K_k\cup K_1)$.

\medskip

\noindent{\bf Subcase 1.2.} There is a path with at least two pairs $(i,j)$ and $(i^\prime,j^\prime)$ satisfying (\ref{i and j}).

\medskip

Since $x_1Px_{i^\prime} x_mPx_{j^\prime}x_1$ is a cycle on at most $2k+1$ vertices, we have $j^\prime-i^\prime\geq m-2k$ and $j-i\leq 3$. Similarly, we get $j-i\geq m-2k $ and $j^\prime-i^\prime\leq 3$. Thus $m\leq 2k+3$.

\medskip

(a). $m=2k+3$. Then $j-i=j^\prime-i^\prime=3$. Clearly $x_1Px_{i} x_mPx_{j}x_1$ and $x_1Px_{i^\prime} x_mPx_{j^\prime}x_1$ are cycles on $2k+1$ vertices.  Without loss of generality, suppose that $d_G(x_{i+1})\geq k$ and $d_G(x_{i+2})\geq k$. Clearly, there is no vertex of $G-V(P)$ which is adjacent to $x_{i+1}$ or $x_{i+2}$. Otherwise there is a path on at least $m+1$ vertices, a contradiction. Hence, $d_C(x_{i+1})\geq k-1$ and $d_C(x_{i+2})\geq k-1$, where $C= x_1Px_{i} x_mPx_{j}x_1=y_1y_2\ldots y_{2k+1}y_1$. Let $y_u=y_v$ where $u\equiv v$ (mod) $2k+1$. For each $y_{q}\in N_C(x_{i+1})$, we have the following:
\begin{equation}\label{eq 1 for lemma}
\mbox{each vertex of } G-V(C) \mbox{ cannot be adjacent to } y_{q-1}\mbox{ and }y_{q+1},
\end{equation}
and
\begin{equation}\label{eq 2 for lemma}
x_{i+2}  \mbox{ is nonadjacent to } y_{q-2},y_{q-1},y_{q+1} \mbox{ or }y_{q+2}.
\end{equation}
Otherwise, there is a path on at least $m+1$ vertices or a cycle of length at least $2k+2$, a contradiction.

Let $1\leq t\leq 2k$. We say that an ordered pair of vertices $(y_{\ell},y_{\ell+t})$ of $C$ {\it adhere to} $x_i$ if $y_{\ell}\in N_C(x_i)$ and $y_{\ell+t}\in N_C(x_i)$ but $y_{\ell+w}\notin N_C(x_i)$ for $w=1,\ldots,t-1$. Since $d_C(x_{i+1})\geq k-1$, by (\ref{eq 1 for lemma}), we consider the following four cases according to
the situations that ordered pairs $(y_{\ell},y_{\ell+t})$ adhere to $x_{i+1}$ with $t\geq3$:

(a.1). There is only one  ordered pair $(y_{\ell_1},y_{\ell_1+3})$ adhering to $x_{i+1}$ ($d_C(x_{i+1})=k$). Then $x_{i+2}$ is nonadjacent to any vertex of $C$ by (\ref{eq 2 for lemma}), a contradiction.

(a.2). There are three ordered pairs  $(y_{\ell_1},y_{\ell_1+3})$, $(y_{\ell_2},y_{\ell_2+3})$ and $(y_{\ell_3},y_{\ell_3+3})$ adhering to $x_{i+1}$ ($d_C(x_{i+1})=k-1$).
By (\ref{eq 1 for lemma}) and (\ref{eq 2 for lemma}), we have $m=11$, i.e., $k=4$, otherwise, $d_G(x_{i+2})\leq k-1$, a contradiction. Without loss of generality, let $x_{i+1}$ and $x_{i+2}$ be adjacent to all of $y_1$, $y_4$ and $y_7$. Since $G$ is 2-connected, each vertex of $G- \{y_1,y_4,y_7\}$ can only be adjacent to $y_1$, $y_4$ and $y_7$, otherwise $c(G)\geq 10$, a contradiction. Thus $G\subseteq K_3+M_{n-3}$.

(a.3). There are two ordered pairs $(y_{\ell_1},y_{\ell_1+3})$ and $(y_{\ell_2},y_{\ell_2+4})$ adhering to $x_{i+1}$ (as $d_C(x_{i+1})=k-1$).
By (\ref{eq 1 for lemma}) and (\ref{eq 2 for lemma}), we have $m=9$, i.e., $k=3$, otherwise, $d_G(x_{i+2})\leq k-1$, a contradiction.
Without loss of generality, let $x_{i+1}$ and $x_{i+2}$ be adjacent to both of $y_1$ and $y_4$. Each vertex of $G- (V(C)\cup\{x_{i+1},x_{i+2}\})$ can only be adjacent to $y_1$, $y_4$ and $y_6$, otherwise there is a path on 10 vertices, contradicting the maximality of $P$. If $y_5$ is adjacent to $y_7$, then each vertex of $G-(V(C)\cup\{x_{i+1},x_{i+2}\})$ can only be adjacent to $y_1$ and $y_4$. Note that each path in $G$ contains at most 9 vertices, by an easy observation, we have $G= K_2+(K_3\cup M_{n-5})$. If $y_5$ is not adjacent to $y_7$, then each non-isolated vertex of $G-(V(C)\cup\{x_{i+1},x_{i+2}\})$  can only be adjacent to $y_1$ and $y_4$. The isolate vertex in $G-(V(C)\cup\{x_{i+1},x_{i+2}\})$ can be adjacent to $y_1$, $y_4$ and $y_6$. Thus  by an easy observation, we have $G=K_2+(S_s\cup M_{n-s-2})$ (mapping $y_6$ to the center of $S_s$).

(a.4) There is an ordered pair $(y_{\ell},y_{\ell+5})$ adhering to $x_{i+1}$ ($d_C(x_{i+1})=k-1$).
Then $x_{i+2}$ is nonadjacent to any vertex of $C$ by (\ref{eq 2 for lemma}), a contradiction.

\medskip

(b). $m=2k+2$. Suppose that there is a cycle, say $C_{2k+1}=y_1y_2\ldots y_{2k+1}y_1$,
on $2k+1$ vertices and a vertex  with degree $k$, say $x$, which does not belong
to this cycle. Then $G-V(C)=\overline{K}_{n-2k-2}$, otherwise there is a path
on at least $m+1$ vertices, a contradiction. Since $c(G)\leq 2k+1$, $x$ cannot
be adjacent to consecutive vertices of $C_{2k+1}$.
Let $N_G(x)=\{y_1,y_4,y_6,\ldots,y_{2k-2},y_{2k}\}$. There is only one
edge $y_2y_3$ in $G[V(C_{2k+1})\setminus N_G(x)]$, otherwise $c(G)\geq 2k+2$,
a contradiction. Moreover, each vertex of $G-C_{2k+1}$ is nonadjacent to
$V(C_{2k+1})\setminus N_G(x)$. Thus $G\subseteq H(n,2k+2,k)$. Now we may
suppose that $j-i=3$, $j^\prime-i^\prime=2$, $n=m=2k+2$ and $d_G(x_{i^\prime+1})<k$.
Moreover, there are only two pairs $(i,j)$ and $(i^\prime,j^\prime)$
satisfying (\ref{i and j}).
Otherwise, consider the pair $(i^{\prime\prime},j^{\prime\prime})$.
There is a cycle on $2k+1$ vertices
and a vertex which does not belong to this cycle with degree at least $k$. Hence,
we are done by the previous argument. Since $j-i=3$, by (\ref{neighbour of end vertices}),
we have $U=V(x_1Px_i)\cup V(x_jPx_{2k+2})$ and $d_P(x_1)=d_P(x_{2k+2})=k$.
Consider the paths $x_{s-1}Px_1x_sPx_{2k+2}$ and $x_{t+1}Px_{2k+2}x_tPx_1$.
The vertex $x_1$ cannot be adjacent to both of consecutive vertices of $V_1$ and $x_{2k+2}$
cannot be adjacent to both of consecutive vertices of $V_2$. Thus we have $i^\prime =k-1$,
$j^\prime=i=k+1$ and $j=k+4$. Moreover, $N_P[x_{\ell}]=N_P[x_1]$ for $2\leq\ell\leq k-2$
and $N_P[x_{\ell}]=N_P[x_{2k+2}]$ for $k+5\leq\ell\leq2k+1.$ Hence,
$x_{k+2}$ is nonadjacent to $V(x_1Px_{k-2})\cup V(x_{k+5}Px_{2k+2})\cup \{x_k\}$,
which implies $d_G(x_{k+2})\leq 4.$
Thus we have $k\leq 4$. By a direct observation ($c(G)\leq 2k+1$), we have $G=H(8,8,3)$ for $k=3$,
and $G=K_3+M_7$ for $k=4$.

\medskip

\noindent{\bf Case 2.} Case 1 does not occur and there exists an $i$ such that $x_i \in N_P(x_1)$ and  $x_i\in N_P(x_m)$.

\medskip

Since $G$ is 2-connected, there exists a path $Q$ with $V(Q)\cap V(P)=\{x_u,x_v\}$ and $1\leq u<i<v\leq m$. Let
$$
p=\min\{w>u:x_w\in N_P(x_1)\} \mbox{ and } q=\max\{w<v:x_w\in N_P(x_{m})\}.
$$
Then $C=x_1Px_uQx_vPx_{m}x_qPx_px_1$ is a cycle containing $\{x_1,x_m\}\cup N_P(x_1)\cup N_P(x_{m})$. Hence
\begin{equation}\label{eq for case 2}
G \mbox{ contains a cycle of length at least }d_P(x_1)+d_P(x_m)+1.
\end{equation}
Since $c(G)\leq2k+1$, by (\ref{eq for case 2}), we have $d_P(x_1)=d_P(x_m)=k$, $x_u$  is adjacent to $x_v$ and
\begin{equation}\label{eq for T(k,r)}
V(x_1Px_u)\cup V(x_pPx_q)\cup V(x_vPx_m)=N_P[x_1]\cup N_P[x_m].
\end{equation}

\noindent{\bf Claim 3}. $p=q=i$.

\medskip

\begin{proof} 
Without loss of generality, suppose that  each vertex of $V(x_1Px_{i-1})$ is of degree at least $k$ in $G$. 
Consider the path $x_{\ell}Px_1x_{\ell+1}Px_m$.
We have $N_P[x_\ell]=N_P[x_1]$  for $2\leq \ell\leq u-1$. 
Otherwise, there is a path on at least $m+1$ vertices or a cycle of length at least $2k+2$, both give us a contradiction. 
Suppose that $p<i$. 
Then by (\ref{eq for T(k,r)}), we have $x_{i-1}\in N_P(x_1)$. 
Thus $P^\ast=x_uPx_{i-1}x_1Px_{u-1}x_iPx_m$ is a path on $m$ vertices such that $x_v\in N_{P^\ast}(x_u)$, $x_i\in N_{P^\ast}(x_m)$ and $x_i$ precedes $x_v$ in $P^\ast$, a contradiction to our assumption ($d_G(x_u)\geq k$, $d_G(x_m)\geq k$ and Case 1 does not occur). Thus, we have $p=i$.

If $d_G(x_v)\geq k$ and $d_G(x_{v+1})\geq k$, then, as the proof of $p=i$, we can similarly show that $q=i$. Thus we may suppose that either $d_G(x_v)\leq k-1$ or $d_G(x_{v+1})\leq k-1$.

Suppose for a contradiction that $q\geq i+1$. First we show that $q=v-1$. We consider the following two cases.

\medskip
(a) $d_G(x_v)\leq k-1$. Consider the path $x_{\ell}Px_mx_{\ell-1}Px_1$. We have $N_P[x_\ell]=N_P[x_m]$ for $v+1\leq \ell\leq m-1$.  Thus we have $q=v-1$, otherwise $x_1Px_ux_vPx_{i+1}x_{v+1}Px_mx_ix_1$ is a cycle on at least $2k+2$ vertices, a contradiction.

\medskip

(b) $d_G(x_{v+1})\leq k-1$. Consider the path $x_{\ell}Px_mx_{\ell-1}Px_1$.
We have $N_P[x_\ell]=N_P[x_m]$ for $v+2\leq \ell\leq m-1$.
If $q<v-2$, then $x_qPx_vx_uPx_1x_ix_{q-1}x_{v+2}Px_mx_q$
is a cycle on at least $2k+2$ vertices. If $q=v-2$, then
$x_{v-1}$ is nonadjacent to any vertex of $G-V(P)$, otherwise
there is a path on at least $m+1$ vertices, a contradiction.
Moreover, $x_{v-1}$ is only adjacent to $x_{v-2}$ and $x_{v}$
of $V(P)$, otherwise it is not hard to see that there is a
cycle on at least $2k+2$ vertices (note that $G[V(x_1Px_{u-1})\cup \{x_{i}\}]$
is a complete graph on $k$ vertices and $N_P[x_\ell]=N_P[x_m]$ for
$v+2\leq \ell \leq m-1$). Thus we have $d_G(x_{v-1})\leq k-1$,
which contradicts that there is at most one vertex of $G$
with degree less than $k$.

\medskip

Now we may suppose that $q=v-1$. Each vertex of $G-V(x_{i}Px_m)$
is nonadjacent to any of $V(x_{i+1}Px_m)$ except for the edge $uv$;
otherwise there is a path on $m+1$ vertices, a cycle on at least
$2k+2$ vertices, or Case 1 occurs, each gives a contradiction.
Hence $d_G(x_v)\geq k$, otherwise as $|V(x_{i+1}Px_m)|=k$, there are at least two vertices with degree less than $k-1$ in $V(x_{i+1}Px_m)$, a contradiction.
Thus Case 1 ($P^\ast=x_vPx_mx_{v-1}Px_1$) occurs, a contradiction.
So we have $q=i$. The proof is complete.\end{proof}

It follows from Claim 3 and the proof of Claim 3 that,
each pair in $V(x_1Px_{u})\cup \{x_i\}$ except for
$x_ux_i$ is an edge of $G$, and each pair of
$V(x_{v}Px_m)\cup \{x_i\}$ except for $x_vx_i$
is an edge of $G$. Hence the longest path starting
from $x_u$ (and $x_v$) and ending at $x_i$ is on
at most $k+2$ vertices. Thus, by the maximality of
$G$, $x_i$ is adjacent to both of $x_u$ and $x_v$.
Hence, there is no component of $G-\{x_u,x_v,x_i\}$
which is adjacent to both of $x_u$ and $x_v$, since otherwise
there is a cycle on at least $2k+2$ vertices. Let $A$
be the union of components of $G-\{x_u,x_v,x_i\}$ which
are adjacent to $x_v$ and $B$ be the union of components
of $G-\{x_u,x_v,x_i\}$ which are adjacent to $x_u$. Since
the longest path starting at $x_i$ through $A$ ending
at $x_v$ is on at most $k+1$ vertices (note that $p=i$) and
$G[V(A)\cup \{x_i,x_v\}]$ is 2-connected, by
Lemma~\ref{minimum degree without a-b path},
$A$ is a subgraph of $s K_{k-1}\cup K_1$
for some $s\geq 1$. Similarly,  $B$ is
a subgraph of $t  K_{k-1}\cup K_1$
for some $t\geq 1$. Since there is
at most one vertex of $G$ with degree less than
$k$, $G=F(s,t,k)$ for some $s\geq 1$ and $t\geq 1$.

\section{Concluding remarks}\label{Sec:Remark}

Our result has implications in extremal graph theory and spectral graph theory.
Moreover, our result can also imply a slight stronger form of a question posed by Fomin, Golovach, Lokshtanov, Panolan, Saurabh and Zebhavi (see  \cite{FGSS-Arxiv}).

\subsection{Some corollaries}
Ali and Staton \cite{AS96} characterized  all graphs on at least $2\delta+1$ vertices
with minimum degree $\delta$ and all longest paths of $2\delta+1$ vertices.

\begin{theorem}[Ali and Staton  \cite{AS96}]
Let $k>1$ and $n\geq 2k+1$. Let $G$ be a graph on $n$ vertices with $\delta(G)\geq k$.
If $G$ is connected then $P_{2k+2}\subset G$, unless $G\subset H(n,2k,k)$,
or $G=K_1+ tK_k$.
\end{theorem}

In the process of solving problems in spectral graph theory, Nikiforov and Yuan \cite{NY14}
obtained a much more complicated stability result.

\begin{theorem}[Nikiforov and Yuan \cite{NY14}]
Let $k\geq 2$ and $n\geq 2k+2$. Let $G$ be a graph on $n$ vertices with $\delta(G)\geq k$.
If $G$ is connected, then $P_{2k+3}\subset G$, unless one of the following holds:\\
(a) $G\subset H(n,2k,k) $;\\
(b) $G=K_1+ tK_k$;\\
(c) $G\subset K_1+ ((t-1)K_k\cup K_{k+1})$;\\
(d) $G$ is obtaining by joining the centers \footnote{The center of $K_1+s K_k$ is the vertex with degree $sk$.}
of $K_1+s K_k$ and $K_1+t K_k$.
\end{theorem}

Let $G$ be a connected graph such that every
vertex except for at most one vertex is of degree at least $k-1$.
Denoted by $G^\ast$ the graph obtained from $G$ by adding a new vertex and joining it to each vertex of $G$.
Note that $G^\ast$ is a 2-connected graph such that every vertex except for at most one vertex is of degree at least $k$. 
With this fact in mind, one can see
Theorem~\ref{Thm:stability-cycle} implies the following theorem, which is a common generalization (we do not need the condition on the number of vertices) of results
of Ali and Staton \cite[Theorem~2]{AS96} and Nikiforov and Yuan \cite{NY14} .

\begin{theorem}\label{Thm:stability-path}
Let $k\geq 1$. Let $G$ be a connected $n$-vertex graph which does not contain a path on $\min\{n,2k+3\}$ vertices.
If every vertex except for at most one vertex is of degree at least $k$, then $G$ is a subgraph of one of the following graphs.
\begin{itemize}
\item  $H(2k+2,2k+1,k)$   and   $H(n,2k+2,k)$ with $n\geq 2k+3$;
\item $K_1+(tK_{k}\cup K_{k+1}\cup K_1)$ with $t\geq 1$;
\item  the graph obtained by joining the centers
$K_1+s K_k$ and $K_1+(t K_k\cup K_1)$, where $s\geq 1$ and $t\geq 1$; \footnote{The center of $K_1+(t K_k\cup K_1)$ is the vertex with degree $tk+1$.}
\item $K_1+ M_{t}$ with $t\geq 6$, $K_1+(S_s\cup M_{t})$  with $ s+t\geq 6$ when $k=2$, and $K_2+ M_{t}$ with $t\geq 7$  when $k=3$.
\end{itemize}
\end{theorem}
\begin{proof}
Let $G^\ast$ be obtained from $G$ by adding a new vertex $v$ and joining it to all vertices of $G$.
Then  every vertex except for at most one vertex of $G^\ast$ is of degree at least $k+1$ and $c(G^\ast)\leq 2k+3$.
Apply Theorem~\ref{Thm:stability-cycle} to $G^\ast$ and remove $v$, the result follows.
\end{proof}

An odd wheel $W_{2k+3}$ is a graph formed by connecting a single vertex to all vertices of a cycle of length $2k+2$.
In \cite{Y22}, the second author proved that  the extremal graphs for $W_{2k+3}$ are obtained by taking an extremal graph for $\{S_{k+2},P_{2k+1}\}$ with order $n/2+o(n)$, say $G_1$, a graph with a unique edge and order $n/2+o(n)$, say $G_2$, and joining each vertex of $G_1$ to each vertex of $G_2$.

Let $\mathcal{H}$ be a given family of graphs. We say a graph is $\mathcal{H}$-free if it does not contain any $H\in \mathcal{H}$ as a subgraph.
The {\it Tur\'{a}n number} of $\mathcal{H}$, denoted by ex$(n,\mathcal{H})$, is the maximum number of edges in an  $n$-vertex $\mathcal{H}$-free graph.
We call an $\mathcal{H}$-free graph with ex$(n,\mathcal{H})$ edges an extremal graph for $\mathcal{H}$.

\begin{proposition}[Yuan \cite{Y22}]\label{k-1 regular no C2k}
Let $n\geq 2k+2$. Then
$$\emph{ex}(n,\{S_{k+2},P_{2k+1}\})= \lfloor kn/2 \rfloor.$$
\end{proposition}

A corollary of Theorem~\ref{Thm:stability-path} is the following which helps us to characterize all extremal graphs for $W_{2k+3}$.

\begin{theorem}
Let $n\geq 2k+2$ and $k\geq 1$. Then the extremal graphs for $\{S_{k+2},P_{2k+1}\}$ consist of components with order at most $2k$.
\end{theorem}
\begin{proof}
Let $G$ be an extremal graph for $\{S_{k+2},P_{2k+1}\}$ on $n$ vertices.
By Proposition~\ref{k-1 regular no C2k}, we have $e(G)=\left\lfloor kn/2 \right\rfloor$, whence there is at most one vertex with degree $k-1$ and all other vertices have degree $k$. 
Let $X$ be a component of $G$ with order at least $2k+1$. 
Clearly, $X$ does not contain a Hamilton path.
Applying Theorem~\ref{Thm:stability-path}, as $G$ is $S_{k+2}$-free, it is easy to check that $X$ cannot be a subgraph of graphs in Theorem~\ref{Thm:stability-path}, a contradiction.
The proof is complete.
\end{proof}

Our result can also imply the stability results for linear forest in \cite{CZ19}.
Since results in \cite{CZ19} are very complicated, we skip the details.
We conclude our paper with the following algorithmic discussions on our result.

\subsection{Algorithmic discussions}
Note that, for large $n$, $H(n,2k+2,k)$, $F_1(t,k)$ and  $F(s,t,k)$ have $k,2$ and $3$ vertices with degree at least $k+2$ respectively, and each of them have at most $k+2$ neighbours with degree more than $k$.
We propose the following algorithm by finding vertices with degree more than $k$ first.

\medskip

\noindent
{\bf Algorithm:} Let $G$ be a 2-connected non-Hamiltonian $n$-vertex graph satisfying that  every vertex except for at most
one vertex is of degree at least $k\geq 5$. Determine whether $c(G)\geq 2k+2$.

\medskip

\noindent
{\bf Input:} Given a $2$-connected graph $G$ on $n$ vertices such that all but at most one vertex of it have degree at least $k$ and $e(G)=\Theta( kn)$\footnote{If $e(G)\geq 2kn$, then by Theorem~\ref{Thm:stability-cycle}, we can easily get $c(G)\geq 2k+2$}.

\medskip

\noindent
{\bf Output:} $c(G)\geq 2k+2$ or $c(G)\in\{2k+1, 2k+2\}$ (when $c(G)\geq 2k+2$ we set $T=1$ and $c(G)\in\{2k+1, 2k+2\}$ we set $T=0$).

\medskip

\noindent Take any $k+3$ vertices of $G$, if none of those vertices
has degree $k$, then set $T=1$; else take a vertex, say $x$, with degree $k$:
  \begin{itemize}
          \item if each vertex of $N(x)$ has degree at least $k+2$,\footnote{$G$ is a subgraph of $H(n,2k+2,k)$.} while there is at most one edge in $E(G-N(x))$, set $T=0$; else set $T=1$;

          \item if there are exactly two vertices, say $u, v$ with $d(u)\geq d(v)$,  of $N(x)$ have degree at least $k+2$,\footnote{$G$ is a subgraph of  $F(s,t,k)$ or $F_1(t,k)$.}
           while  \begin{itemize}
           \item  $G$ is a subgraph of  $F_1(s,k)$ or 
           \item  $G$ is a subgraph of  $F(t_1,t_2,k)$, 
             \end{itemize}
             set $T=0$; else set $T=1$;
          \item else, we set $T=1$.
        \end{itemize}

%\begin{algorithm}[H]
%        \caption{How to write algorithms}
%        \KwIn{this text}
 %       \KwOut{how to write algorithm with \LaTeX2e }
%
%        Set $s=1$, $t_1=t_2=0$ and $V=0$; take $k+3$ vertices of $G$, \\
%         \eIf{none of those vertices has degree $k$}
 %        {set $V=1$}{set $V=0$;\\
%        \While{$V=0$}{take a vertex, say $x$, with degree $k$;\\
%            \eIf{each vertex of $N(x)$ has degree at least $k+2$,}
%            {\While{$s<n$}{taking an vertex $v_s$ of $G\setminus N[X]$;\\
%                  \begin{itemize}
 %                  \item    if $v_s$ is not adjacent to $N(x)$, then replace $t_1$ by $t_1+1$, $s$ by $s+1$, and we set $V=1$ and stop when $t_1\geq 2$.
  %                 \item    else if replace $s$ by $s+1$;
 %                  \end{itemize}
  %          end while;}
   %         {}
  %          {
  %              go back to the beginning of current section; \\
  %          }
%    }}
%\end{algorithm}

Conclusion: For any $k+3$ vertices, if there is no
vertex with degree $k$, then stop our algorithm within $O(kn)$ unit operations.
Let $x$ be a vertex with degree $k$ in $G$.
Note that there are $\Theta(kn)$ edges in $G$.
If each vertex of $N(x)$ has degree at least $k+2$, then
we will stop our algorithm within $O(kn)$ unit operations;
if there are exactly two vertices of $N(x)$ have degree at least
$k+2$, then we will stop within $O(kn)$  unit operations;
else we will stop our algorithm within $O(kn)$ unit operations.
Thus, the above algorithm has worst case complexity
$O(kn)$.

\end{document}